\documentclass[a4paper,12pt]{amsart}
\usepackage{amsmath, amsfonts, amsthm, amssymb, mathrsfs, pb-diagram, graphics}
\usepackage{verbatim, color, tikz, fullpage}
\usetikzlibrary{patterns}

\newtheorem*{main}{Main Theorem}
\newtheorem*{popo}{Popoviciu's Theorem}
\newtheorem{lem}{Lemma}
\newtheorem{cor}{Corollary}

\usepackage{hyperref}
\hypersetup{colorlinks,% 
citecolor=red,% 
linkcolor=blue,% 
}
\usepackage{color}

\newcommand{\Z}{{\Bbb Z}}

\newcommand{\R}{{\Bbb R}}

\newcommand{\eps}{\epsilon}
\newcommand{\om}{\omega}
\newcommand{\G}{\mathcal{G}}
\newcommand{\Ge}{\G_{\eps}}
\newcommand{\Geh}{\hat{\G}_{\eps}}

\newcommand{\PP}{\mathcal{P}}
\newcommand{\ONE}{\mathbf{1}}

\DeclareMathOperator{\vol}{vol}
\DeclareMathOperator{\proj}{Proj}

\DeclareMathOperator{\interior}{int}

\begin{document}

\title{Macdonald's solid-angle sum for real dilations of rational polygons}

%\date{\today}
\makeatletter
\@namedef{subjclassname@2010}{\textup{2010} Mathematics Subject Classification}
\makeatother

\author{Quang-Nhat Le}
\address{Department of Mathematics, Brown University, 
Box 1917, 151 Thayer Street, Providence, RI 02912}
\email{quang\_nhat\_le@brown.edu}

\author{Sinai Robins}
%\thanks{The second author is grateful for the support of ICERM, at Brown University}
\address{Instituto de Matematica e Estatistica, Universidade de S\~ao Paulo, Rua do Matao 1010, 05508-090 S\~ao Paulo, Brazil}
\email{sinai\_robins@brown.edu}

\address{Department of Mathematics, Brown University, 
	Box 1917, 151 Thayer Street, Providence, RI 02912}
\email{sinai\_robins@brown.edu}

\begin{abstract}
The solid-angle sum $A_{\PP}(t)$ of a rational polytope ${\PP}\subset \mathbb R^d$, with $t \in \Z$ was first investigated by I.G. Macdonald.  
Using our Fourier-analytic methods, developed in  \cite{DLR}, we are able to establish an explicit formula for $A_{\PP}(t)$, for any real dilation $t$ and any rational polygon $\PP \subset \mathbb R^2$. Our formulation sheds additional light on previous results, for lattice-point enumerating functions of triangles, which are usually confined to the case of integer dilations.   Our approach differs from that of 
Hardy and Littlewood in $1992$ \cite{HL1}, but offers an alternate point of view for enumerating weighted lattice points in real dilations of real triangles. 
\end{abstract}

\maketitle

\section{Introduction}

In his pioneering papers \cite{Mac1} and \cite{Mac2}, I.G. Macdonald introduced a weighted lattice-point sum of polytopes which resembles the Ehrhart function in many ways but has some useful and elegant additional properties. Given a closed polytope $\PP$ and a real number $t$, Macdonald's solid-angle sum counts weighted lattice points inside the dilation $t\PP := \{ tx : x \in \PP \}$, with weights being the solid angles subtended at each lattice point. In this article, we will restrict ourselves to the case of polygons in $\mathbb R^2$. In this setting, the solid angle at a point $x$, which will be denoted as $\omega_{\PP}(x)$, is defined as follows:
	\begin{align*}
		\om_P(x) := \left \{ 
		\begin{array}{ll}
		1 & \textup{if } x \in {\interior(\mathcal{P}}), \\
		0 & \textup{if } x \notin   {\mathcal{P}}, \\
		1/2 & \textup{if } x \textup{ lies in the interior of an edge of } {\mathcal{P}}, \\
		{\theta_x / 2\pi} & \textup{if } x \textup{ is a vertex of } {\mathcal{P}},
		\end{array}
		\right.
	\end{align*}
where $\theta_x$ is the angle, measured in radians, at a vertex $x$ of ${\mathcal{P}}$.  We note that this quantity can be defined more generally as the solid angle of the tangent cone of $\PP$ at $x$. Now we can define the solid-angle sum of the polygon $\PP$, following Macdonald, by
	$$A_{\mathcal{P}}(t) := \sum_{x\in\Z^2} \om_{t{\PP}}(x).$$   

The solid-angle sum is closely related to Ehrhart's integer-point sum, which is defined as
$$L_{\mathcal{P}}(t) := \sum_{x\in\Z^2} \ONE_{t{\PP}}(x),$$
where $\ONE_{S}(x)$ is the indicator function of the set $S$. The integer-point sum simply counts the number of lattice points inside $t\PP$ without any weights. 

On the surface, the Ehrhart sum seems to have a slightly more natural definition, but the solid-angle sum enjoys nice properties that the Ehrhart sum does not possess.  

First, the solid-angle sum has a strong additive property (also known as a {\it simple valuation}), namely:
	\begin{equation*}\label{AddProp}
		A_{{\mathcal{P}}_1}(t) + A_{{\mathcal{P}}_2}(t) = A_{{\mathcal{P}}_1 \cup {\mathcal{P}}_2}(t),
	\end{equation*}
which is valid for any polytopes ${\mathcal{P}}_1$, ${\mathcal{P}}_2$ whose interiors are disjoint, whereas a similar formula for the integer-point sum has to take into account the common boundary of $\PP_1$ and $\PP_2$.
	
	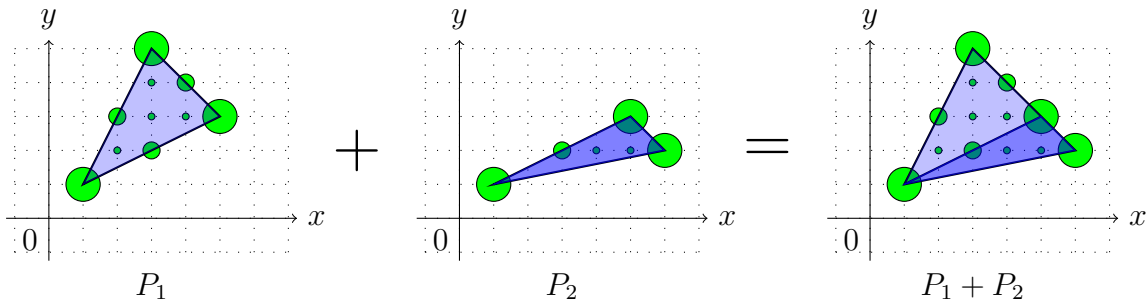
\begin{figure}[!h]
		\centering
		\begin{tikzpicture}[scale=0.45]
			\draw (0,0) node[below left] {$0$};
			\draw[loosely dotted] (-1,-1) grid (7,5);
			\draw[->] (-1.25,0) -- (7.25,0) node[right] {$x$};
			\draw[->] (0,-1.25) -- (0,5.25) node[above] {$y$};
			\draw[fill = green] (3,4) circle (.1cm);
			\draw[fill = green] (4,3) circle (.1cm);
			\draw[fill = green] (3,3) circle (.1cm);
			\draw[fill = green] (3,2) circle (.25cm);
			\draw[fill = green] (2,2) circle (.1cm);
			\draw[fill = green] (2,3) circle (.25cm);
			\draw[fill = green] (4,4) circle (.25cm);
			\draw[fill = green] (5,3) circle (.5cm);
			\draw[fill = green] (3,5) circle (.5cm);
			\draw[fill = green] (1,1) circle (.5cm);
			\draw[thick] (1,1) -- (5,3) -- (3,5) -- cycle;
			\filldraw[nearly transparent, blue] (1,1) -- (5,3) -- (3,5) -- cycle;
			\draw (3,-2) node {$P_1$};
			
			\draw (9,2) node[scale = 2] {$+$};
			
			\draw (0+12,0) node[below left] {$0$};
			\draw[loosely dotted] (-1+12,-1) grid (7+12,5);
			\draw[->] (-1.25+12,0) -- (7.25+12,0) node[right] {$x$};
			\draw[->] (0+12,-1.25) -- (0+12,5.25) node[above] {$y$};
			\draw[fill = green] (5+12,2) circle (.1cm);
			\draw[fill = green] (4+12,2) circle (.1cm);
			\draw[fill = green] (3+12,2) circle (.25cm);
			\draw[fill = green] (5+12,3) circle (.5cm);
			\draw[fill = green] (6+12,2) circle (.5cm);
			\draw[fill = green] (1+12,1) circle (.5cm);
			\draw[thick] (1+12,1) -- (6+12,2) -- (5+12,3) -- cycle;
			\filldraw[semitransparent, blue] (1+12,1) -- (5+12,3) -- (6+12,2) -- cycle;
			\draw (3+12,-2) node {$P_2$};

			\draw (9+12,2) node[scale = 2] {$=$};

			\draw (0+24,0) node[below left] {$0$};
			\draw[loosely dotted] (-1+24,-1) grid (7+24,5);
			\draw[->] (-1.25+24,0) -- (7.25+24,0) node[right] {$x$};
			\draw[->] (0+24,-1.25) -- (0+24,5.25) node[above] {$y$};
			\draw[fill = green] (3+24,4) circle (.1cm);
			\draw[fill = green] (4+24,3) circle (.1cm);
			\draw[fill = green] (3+24,3) circle (.1cm);
			\draw[fill = green] (5+24,2) circle (.1cm);
			\draw[fill = green] (4+24,2) circle (.1cm);
			\draw[fill = green] (3+24,2) circle (.25cm);
			\draw[fill = green] (2+24,2) circle (.1cm);
			\draw[fill = green] (2+24,3) circle (.25cm);
			\draw[fill = green] (4+24,4) circle (.25cm);
			\draw[fill = green] (5+24,3) circle (.5cm);
			\draw[fill = green] (6+24,2) circle (.5cm);
			\draw[fill = green] (3+24,5) circle (.5cm);
			\draw[fill = green] (1+24,1) circle (.5cm);
			\draw[thick] (1+24,1) -- (6+24,2) -- (3+24,5) -- cycle;
			\draw[thick] (1+24,1) -- (5+24,3);
			\filldraw[nearly transparent, blue] (1+24,1) -- (5+24,3) -- (3+24,5) -- cycle;
			\filldraw[semitransparent, blue] (1+24,1) -- (5+24,3) -- (6+24,2) -- cycle;
			\draw (3+24,-2) node {$P_1 + P_2$};
			
		\end{tikzpicture}
		\caption{Additive property of Macdonald's solid-angle sum}
	\end{figure}

Furthermore, Macdonald's solid-angle sum is a better approximation to the continuous volume of $t{\mathcal{P}}$ than Ehrhart's integer-point sum, a claim we can make more precise, as follows.  From the definition of the solid-angle sum, this finite sum associates smaller weights to those lattice points that are contained in the lower-dimensional faces of $\PP$, which offers some initial intuition.  More precisely, when $\PP$ is a lattice polytope and $t$ is an integer, $A_{\PP}(t)$ is an even polynomial (see \cite{Mac1}) with the constant coefficient equal to $0$ and the leading one equal to $\vol(\PP)$. The fact that the codimension-$1$ coefficient of  $A_{\PP}(t)$ vanishes, in addition to the vanishing of half the coefficients of 
 $A_{\PP}(t)$,  indicates that $A_{\PP}(t)$ is a very good approximation to the volume $\vol(t\PP)$, especially for large $t$.  In two dimensions, $A_{\PP}(t)$ is precisely 
 $\vol(t\PP)$, for integer polygons $\PP$,  a fact which is easily equivalent to Pick's theorem.

The solid-angle sum possesses other properties that are shared with the integer-point sum (also known as the Ehrhart polynomial)  $L_{\mathcal{P}}(t)$.  
Ehrhart and Macdonald proved that, for an integer variable $t$,  $A_{\mathcal{P}}(t)$ is a polynomial if $\PP$ is a lattice polytope, and a quasi-polynomial if $\PP$ is a rational polytope; the same is true for the Ehrhart sum.   Moreover, the solid-angle sum also enjoys a reciprocity law which is reminiscent of Ehrhart's Reciprocity Law, both of which were first proved in the general rational polytope case by Macdonald.
More information about these now-classical topics may be found in \cite{Ehr} and \cite{BeR1}, for example. 
	
\begin{main}
		Let $\Delta$ be the triangle with vertices at the origin and the two points $(h,0)$ and $(0,k)$, where $h,k$ are two coprime positive integers. Then, for any nonzero real number $t$, the solid-angle sum of $\Delta$ has the following explicit formula:
		\begin{align*}
			A_{\Delta}(t) &= \frac{hk}{2}t^2-\bar{B}_1(hkt)t + \frac{1}{2hk}\left(\bar{B}_2(hkt)+\frac{h^2+k^2}{6}\right) \\
						&\quad -s(h,k;ht,0)-s(k,h;kt,0) - \frac{\arctan(h/k)}{2\pi}\mathbf{1}_{\Z}(ht) - \frac{\arctan(k/h)}{2\pi}\mathbf{1}_{\Z}(kt),
		\end{align*}
	where $\bar{B}_1(x)$ and $\bar{B}_2(x)$ are the first and second periodic Bernoulli polynomials, $s(h,k;x,y)$ is the Dedekind-Rademacher sum, and $\ONE_{\Z}(x)$ denotes the indicator function of the set of integers $\Z$.
\end{main}
	
\medskip \noindent
We recall here the standard definitions of the Dedekind-Rademacher sums,  for the sake of the reader, and a slightly non-standard definition of the first two Bernoulli polynomials, which we will find very useful due to their compact support.  We define
 the first Bernoulli polynomial by
\begin{equation}
		B_1(x) := \left\{
		\begin{array}{ll}
		x - \frac{1}{2} & \textup{when } x \in (0,1), \\
		0 & \textup{otherwise}.
		\end{array}
		\right.
\end{equation} 
		
The periodized version of $B_1(x)$ is standard, and is defined by $\bar{B}_1(x):= B_1( x-\lfloor x \rfloor)$. It is also known as
 the sawtooth function, or the first periodic Bernoulli polynomial.
We define the second Bernoulli polynomial by 
\begin{equation}
		B_2(x) := \left\{
		\begin{array}{ll}
		 x^2-x+\frac{1}{6} & \textup{when } x \in [0,1], \\
		0 & \textup{otherwise}.
		\end{array}
		\right.
\end{equation}

The periodization of $B_2(x)$, which is also more standard and often called the second periodic Bernoulli polynomial, is defined to be 
$\bar{B}_2(x):= B_2( x-\lfloor x \rfloor)$.   The Dedekind-Rademacher sum is defined by 
\begin{equation}\label{DedekindRademacherSum}
s(h,k;y,x) := \sum_{r \textup{ mod } k} \bar{B}_1\left(h\frac{r+x}{k}+y \right) \bar{B}_1\left(\frac{r+x}{k} \right),
\end{equation}
for any coprime positive integers $h,k$ and any real numbers $x,y$. When $x,y$ are both zero, the 
Dedekind-Rademacher sum reduces to the classic Dedekind sum $s(h,k)$.  The Dedekind-Rademacher sum enjoys a reciprocity law that helps us calculate the sum \eqref{DedekindRademacherSum} in linear time (see \cite{Car}).

We modularize the computations involved in proving the Main Theorem into the next four sections, and the final steps of the proof of the Main Theorem appears in section \ref{Calc3}.

A new method, employing Fourier analysis, was used recently by Diaz, Le, and Robins \cite{DLR} to prove that for an integer polytope ${\mathcal{P}}$, the solid-angle sum $A_{\mathcal{P}}(t)$ takes a polynomial form whose coefficients are periodic functions in the nonzero real variable $t$.  It is easy to see that if we extend the dilation factor $t$ to any nonzero real number, there is no distinction between the case of a lattice polytope $\mathcal{P}$ and  the slightly more general case of a rational polytope $\mathcal{P}$. Thus we may henceforth consider any rational polygon as a dilation of an integer polygon, so that we always conduct our analysis with integer polygons, and their real dilations.

The methodology of \cite{DLR}, although computationally complex in the most general case, turns out to be extremely useful for low-dimensional cases.  The main result above uses this Fourier-analytic machinery to find an {\it explicit formula } for the solid-angle sum $A_{\mathcal{P}}(t)$, for any nonzero  real dilation of an integer polygon in $\R^2$.  Throughout the paper, unless otherwise stated, $t$ is a nonzero real number.
	
	In Section \ref{Breakdown}, we reduce the computation of the solid-angle sum of any integer polygon to the simplest case of a right-angled triangle with a vertex at the origin.  Even this simple case poses considerable difficulties, as is shown in sections \ref{Calc1}, \ref{Calc2} and \ref{Calc3}.    As was shown in \cite{DLR}, $A_{\mathcal{P}}(t)$ is a quasi-polynomial -- a polynomial whose coefficients are periodic functions of $t$.  
We call these coefficients quasi-coefficients. We carry out the detailed calculations of the periodic quasi-coefficients of $A_{\mathcal{P}}(t)$, for the special case of a right triangle, in the three sections \ref{Calc1}, \ref{Calc2} and \ref{Calc3}. The final results involve periodic Bernoulli polynomials and 
Dedekind-Rademacher sums.    Finally, in Section \ref{Apps}, we briefly discuss some classical implications of the explicit formula given by our main result. We also give an analogous formula for the corresponding Ehrhart quasi-polynomial of the right triangle, we point out a connection to the work of Donald Knuth on the Dedekind-Rademacher sums, and we discuss 
further directions for low-dimensional solid-angle sums.

\medskip
{\bf Acknoledgement}.   The second author is grateful for the partial support of FAPESP grant Proc. 2103 / 03447-6, Brazil, and both authors are grateful for the support of ICERM, at Brown University.  The first author would like to express his deepest gratitude to Richard E. Schwartz for his encouragement on this project.  Both authors were stimulated by many interesting conversations with Ricardo Diaz, whom they would like to thank here.

	\section{A specific case}\label{Breakdown}
	
	In this section, we work exclusively with rational polygons.  Suppose that we have obtained an exact formula of $A_{\mathcal{P}}(t)$ when 
	${\mathcal{P}}$ is the right-angled triangle whose vertices include the origin and two points on the coordinate axes. We call this collection of triangles {\bf simple pointed triangles}. If we apply a unimodular transformation $M \in SL(2,\Z)$ to the whole Euclidean plane, we obtain a new triangle $M{\mathcal{P}}$. Clearly, $M$ preserves both the ambient integer lattice $\Z^2$ as well as the face structure of $\PP$. Thus, $M$ maps integer points in the interior of the triangle ${\mathcal{P}}$ bijectively to those in the interior of $M{\mathcal{P}}$ and integer points in the interior of an edge of ${\mathcal{P}}$ bijectively to those in the interior of the corresponding edge of $M{\mathcal{P}}$. Therefore, we can easily compute $A_{M{\mathcal{P}}}(t)$ from $A_{\mathcal{P}}(t)$ by taking care of the normalized angles at the vertices of ${\mathcal{P}}$ and $M{\mathcal{P}}$. 
	
	Hence, our basic case is a triangle which has one of its vertices at the origin, and whose tangent cone at the origin is a unimodular cone. We will call this type of triangle a  {\bf unimodular pointed triangle}.  We remark that the other two tangent cones of a unimodular pointed triangle, located at the vertices which are not the origin, may very well be non-unimodular.

	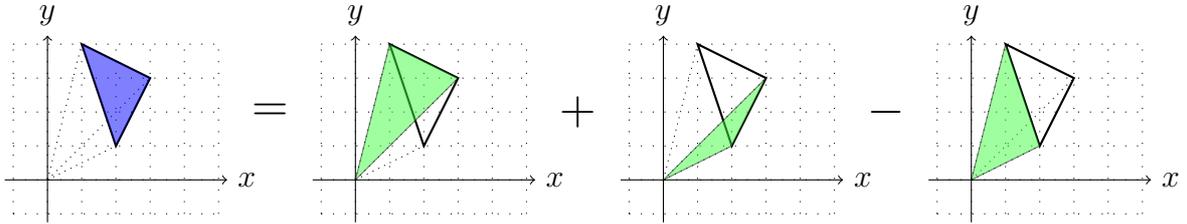
\begin{figure}[!h]\label{fig:Fig1}
		\centering
		\begin{tikzpicture}[scale=0.45]
			\draw[loosely dotted] (-1,-1) grid (5,4);
			\draw[->] (-1.25,0) -- (5.25,0) node[right] {$x$};
			\draw[->] (0,-1.25) -- (0,4.25) node[above] {$y$};
			\draw[thick] (1,4) -- (2,1) -- (3,3) -- cycle;
			\draw[dotted] (0,0) -- (1,4);
			\draw[dotted] (0,0) -- (2,1);
			\draw[dotted] (0,0) -- (3,3);
			\draw[fill = blue, semitransparent] (1,4) -- (2,1) -- (3,3) -- cycle;
			
			\draw (6+0.5,2) node[scale = 1.5] {$=$};
			
			\draw[loosely dotted] (-1+9,-1) grid (5+9,4);
			\draw[->] (-1.25+9,0) -- (5.25+9,0) node[right] {$x$};
			\draw[->] (0+9,-1.25) -- (0+9,4.25) node[above] {$y$};
			\draw[thick] (1+9,4) -- (2+9,1) -- (3+9,3) -- cycle;
			\draw[dotted] (0+9,0) -- (1+9,4);
			\draw[dotted] (0+9,0) -- (2+9,1);
			\draw[dotted] (0+9,0) -- (3+9,3);
			\draw[fill = green!75, semitransparent] (1+9,4) -- (0+9,0) -- (3+9,3) -- cycle;

			\draw (6+9+0.5,2) node[scale = 1.5] {$+$};
			
			\draw[loosely dotted] (-1+18,-1) grid (5+18,4);
			\draw[->] (-1.25+18,0) -- (5.25+18,0) node[right] {$x$};
			\draw[->] (0+18,-1.25) -- (0+18,4.25) node[above] {$y$};
			\draw[thick] (1+18,4) -- (2+18,1) -- (3+18,3) -- cycle;
			\draw[dotted] (0+18,0) -- (1+18,4);
			\draw[dotted] (0+18,0) -- (2+18,1);
			\draw[dotted] (0+18,0) -- (3+18,3);
			\draw[fill = green!75, semitransparent] (0+18,0) -- (2+18,1) -- (3+18,3) -- cycle;

			\draw (6+18+0.5,2) node[scale = 1.5] {$-$};
			
			\draw[loosely dotted] (-1+27,-1) grid (5+27,4);
			\draw[->] (-1.25+27,0) -- (5.25+27,0) node[right] {$x$};
			\draw[->] (0+27,-1.25) -- (0+27,4.25) node[above] {$y$};
			\draw[thick] (1+27,4) -- (2+27,1) -- (3+27,3) -- cycle;
			\draw[dotted] (0+27,0) -- (1+27,4);
			\draw[dotted] (0+27,0) -- (2+27,1);
			\draw[dotted] (0+27,0) -- (3+27,3);
			\draw[fill = green!75, semitransparent] (1+27,4) -- (2+27,1) -- (0+27,0) -- cycle;
			
		\end{tikzpicture}
		\caption{A signed decomposition of a general triangle into pointed triangles}
	\end{figure}
	
	Next, we utilize a result of Barvinok \cite{BaP}, stating that we can decompose an arbitrary rational cone into unimodular cones in polynomial time. Therefore, using Barvinok's algorithm, we can compute in polynomial time the solid-angle sum of any {\bf pointed triangle}, which is defined to be a triangle that has the origin as a vertex.  Finally, we can solve the most general case where ${\mathcal{P}}$ is any rational polygon.  Due to the additive property (\ref{AddProp}), it is easily seen that the solid-angle sum $A_{\mathcal{P}}(t)$ is a signed sum of $A_{\mathcal{P}_k}(t)$ where $\mathcal{P}_k$ are pointed triangles each of which includes two consecutive vertices of ${\mathcal{P}}$ as their own vertices. Therefore, it is clear that, for any integer polygon $\mathcal{P}$, we can calculate the univariate function $A_{\mathcal{P}}(t)$ in polynomial time.  
	
	Summarizing the discussion above, we easily see that we can decompose each rational polygon $\mathcal{P}$ into polynomially many unimodular pointed triangles.  We call this decomposition the {\bf unimodular decomposition} of $\mathcal{P}$.
	
	Because of the above reduction arguments, we only need to focus on simple pointed triangles. Moreover, as noted in the Introduction, it is sufficient to consider only \textit{lattice} simple pointed triangle. Throughout the paper, we assume that $\Delta$ is the simple pointed triangle with three vertices $V_1(0,0)$, $V_2(h,0)$ and $V_3(0,k)$, where $h,k$ are coprime positive integers. We also use $E_1,E_2,E_3$ to denote
	the edges of $\Delta$.  Figure 3 below will help illustrate the above notations.
	
	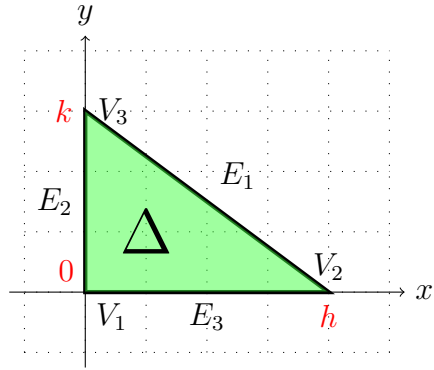
\begin{figure}[!h]\label{righttrianglefigure}
		\centering
		\begin{tikzpicture}[scale=0.8]
			\draw (0,0) node[red,above left] {$0$} node[below right] {$V_1$};
			\draw[loosely dotted] (-1,-1) grid (5,4);
			\draw[->] (-1.25,0) -- (5.25,0) node[right] {$x$};
			\draw[->] (0,-1.25) -- (0,4.25) node[above] {$y$};
			\draw[ultra thick] (0,0) -- node[left] {$E_2$} (0,3) node[red,left] {$k$} node[right] {$V_3$} -- node[above right] {$E_1$} (4,0) node[red,below] {$h$} node[above] {$V_2$} -- cycle; 
			\draw[fill = green!75, semitransparent] (0,3) -- (4,0) -- (0,0);
			\draw (4,0) -- node[below] {$E_3$} (0,0);   
			\draw (1,1) node[scale = 2] {$\Delta$};
		\end{tikzpicture}
		\caption{The simple pointed triangle $\Delta$}
	\end{figure}
	
	Following the notation of \cite{DLR}, we first construct the face poset 
	$\mathbf{G}_{\Delta}$ of the triangle ${\Delta}$, which can also be considered as a directed graph. We first briefly recall the terminology used in \cite{DLR}, keeping in mind that the theory developed there arose naturally by first using Stokes' formula to rewrite the Fourier-Laplace transform of the indicator function of a polytope as a finite sum of weighted Fourier-Laplace transforms of its facets, and then iterating this procedure on each of its facets.

	In the graph $\mathbf{G}_{\Delta}$, each node represents a face of $\Delta$, and each arc
	between two nodes represents the inclusion of one face of $\Delta$ in a larger face of $\Delta$.    Moreover, for each arc $(F,G)$ with $G \subset F$  in the graph 
	$\mathbf{G}_{\Delta}$, we assign a weight, namely the function
		$$W_{(F,G)}(\xi)=\frac{-1}{2\pi i} \frac{\langle \proj_F(\xi), N_F(G) \rangle }{\|\proj_F(\xi)\|^2},$$
	where
	\begin{itemize}
		\item[$\bullet$] $\proj_F(\xi)$ denotes the projection of the vector $\xi$ on the affine space spanned by the face $F$. The result is a vector whose two endpoints are the corresponding projections of two endpoints of the vector $\xi$.
		\item[$\bullet$] $N_F(G)$ is the (unique) outward-pointing unit normal vector of $G$ which resides in the affine space spanned by $F$.
		\item[$\bullet$] $\langle \cdot,\cdot \rangle$ and $\|\cdot\|$ denote the standard inner product and the standard Euclidean norm on $\R^2$.
	\end{itemize}

	\begin{figure}[!h]
		\centering
		\begin{tikzpicture}[scale = 0.75]
			\draw (4,5+1) -- (4,5+9/4+1) -- (4+3,5+1) -- cycle;
			\draw (4,5+1) node {$\times$};
			\draw (4,5+9/4+1) node {$\circ$};
			\draw (4+3,5+1) node {$\ast$};
			\draw (4+0.75,5+0.75+1) node {$\Delta$};
			
			\draw[->,dashed] (4.25,4.75+0.5) -- (2.5-0.25,3.5+0.25);
			\draw (1,1) -- node[left] {$E_2$} (1,1+9/4);
			\draw (1,1) node {$\times$};
			\draw (1,1+9/4) node {$\circ$};
			
			\draw[->,dashed] (5.5,4.75+0.5) -- (5.5,2.25);
			\draw (4,1) -- node[above] {$E_3$} (4+3,1);
			\draw (4,1) node {$\times$};
			\draw (4+3,1) node {$\ast$};
			
			\draw[->,dashed] (6.75,4.75+0.5) -- (8.5+0.5,3.5+0.25);
			\draw (8+2,1+9/4) -- node[above right] {$E_1$} (8+3+2,1);
			\draw (8+2,1+9/4) node {$\circ$};
			\draw (8+3+2,1) node {$\ast$};
			
			\draw[->,dashed] (1,0.75-0.25) -- (1,-1.5);
			\draw[->,dashed] (5,0.75-0.25) -- (1.5,-1.5);
			\draw (1,-2) node {$\times$} node[below] {$V_1$};
			
			\draw[->,dashed] (1.5,0.75-0.25) -- (5,-1.5);
			\draw[->,dashed] (10.5+2,0.75-0.25) -- (6,-1.5);
			\draw (5.5,-2) node {$\circ$} node[below] {$V_3$};
			
			\draw[->,dashed] (11+2,0.75-0.25) -- (11+2,-1.5);
			\draw[->,dashed] (6,0.75-0.25) -- (10.5+2,-1.5);
			\draw (11+2,-2) node {$\ast$} node[below] {$V_2$};
		\end{tikzpicture}
	\caption{The partially ordered set of face containments in the triangle $\Delta$.}
	\end{figure}
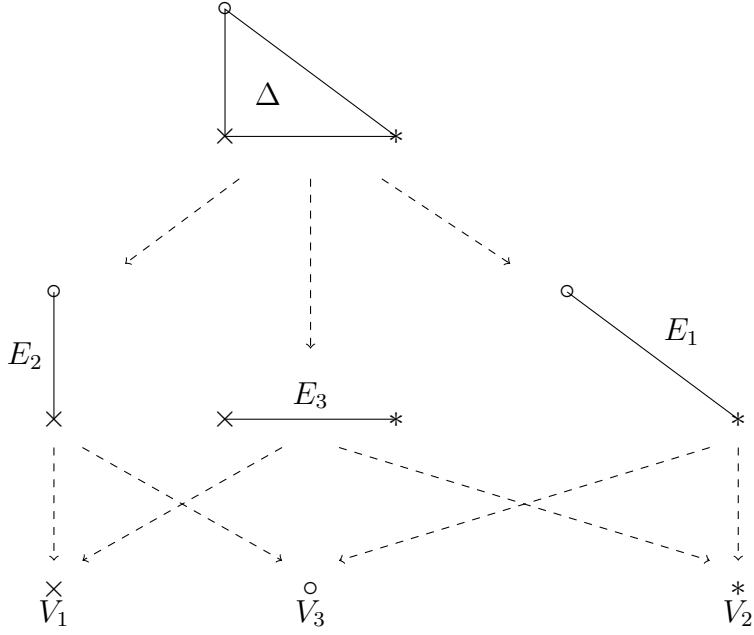

When iterating Stokes' theorem, applied to the exponential function integrated over $P$, in \cite{DLR}, we are naturally led to work with chains in the face poset $\mathbf{G}_{\Delta}$.    

For us, all chains $\mathbf{T}$  in the poset $\mathbf{G}_{\Delta}$ will begin from the root of the poset, namely 
$\Delta$ itself.  We will sometimes use the following notation for these chains of faces.  
Suppose $Y,Z$ are the faces corresponding to the last two nodes of a chain $\mathbf{T}$, so we may write
$\mathbf{T} = (\Delta \to ... \to Y \to Z)$. 
	
	We define the admissible set $S(\mathbf{T})$ of $\mathbf{T}$ to be the set of all points of $\R^2$ that are orthogonal to $Z$ but not $Y$. Finally, we define the following weights associated to the chain $\mathbf{T}$:
	
	\bigskip
	\begin{itemize}
		\item[$\bullet$] The rational weight $\mathcal{R}_{\mathbf{T}}(\xi) = \mathcal{R}_{(\Delta \to ... \to Y \to Z)}(\xi)$ is defined to be the product of weights associated to all of the arcs in $\mathbf{T}$, times the (usual Hausdorff) volume of the last node of the chain $\mathbf{T}$.
		
		\bigskip
		\item[$\bullet$] The exponential weight $\mathcal{E}_{\mathbf{T}}(\xi) = \mathcal{E}_{(\Delta \to ... \to Y \to Z)}(\xi)$ is equal to the evaluation of $e^{-2\pi i\langle\xi,x\rangle}$ at any point $x$ on the affine space spanned by $Z$. Notice that the inner product $\langle\xi,x\rangle$ does not depend on the position of $x$ on that affine space.
		
		\bigskip
		\item[$\bullet$] The total weight $W_{\mathbf{T}}(\xi) = W_{(\Delta \to ... \to Y \to Z)}(\xi)$ assigned to any chain $\mathbf T$
		 is defined to be 
\begin{equation}
W_{\mathbf{T}}(\xi) := \mathcal{R}_{\mathbf{T}}(\xi) \mathcal{E}_{\mathbf{T}}(\xi) \mathbf{1}_{S(\mathbf{T})}(\xi),
\end{equation}		
where $\mathbf{1}_{S(\mathbf{T})}(\xi)$ is the indicator function of the admissible set $S(\mathbf{T})$ of $\mathbf{T}$.
	\end{itemize}
	
	\bigskip
	\noindent
	Therefore, by the Main Theorem of \cite{DLR}, the solid-angle sum of $\Delta$ is
		$$A_{\Delta}(t) = a_2(t) t^2 + a_1(t) t + a_0(t),$$
	for all nonzero values of $t$, where
	\begin{align}
		a_2(t) &= \vol(\Delta) = \frac{hk}{2}, \nonumber \\
		a_1(t) &= t \sum_{i=1}^3 \lim_{\eps\to 0^+} \sum_{\xi\in (E_i^{\perp} \cap \Z^2) \setminus {(0,0)} } W_{(\Delta \to E_i)}(t\xi) \Geh(\xi), \label{Eq:a_1}\\ 
		a_0(t) &= t^2 \sum_{\substack{i,j=1 \\ i \neq j}}^3 \lim_{\eps\to 0^+} \sum_{\xi\in \Z^2 \setminus E_i^{\perp} } W_{(\Delta \to E_i \to V_j)}(t\xi) \Geh(\xi). \label{Eq:a_0}
	\end{align}
	
Here, $\Ge(x)=\eps^{-1}e^{-\pi\|x\|^2/\eps}$ is the $2$-dimensional heat kernel, whose Fourier transform is 	
\begin{equation}
\Geh(\xi)=e^{-\eps\pi\|x\|^2}.
\end{equation}

For ease of reading, we will also use the notation $\mathcal{F} \{f\}$ for the Fourier transform, particularly when the function $f$ becomes too lengthy.
 We notice that $a_1(t)$ and $a_0(t)$ are the aggregations of certain limiting sums. The inner sums in (\ref{Eq:a_1}) and (\ref{Eq:a_0}) are taken over the corresponding admissible sets, thus we can omit the indicator functions in those weights. 
	
Our main tool is the Poisson summation formula, but we employ some other tricks in order to calculate the ensuing limits of infinite lattice sums that arise naturally from Poisson summation.  In fact,  as a global road-map, we first use Poisson summation applied to a smoothed version of the indicator function of $\mathbf P$, and we think of $\mathbf P$ as living in the ``spatial domain'', while the right hand side of Poisson summation allows us to compute infinite lattice sums in the ``frequency domain'';  this infinite integer lattice sum in the frequency domain then breaks up into a finite number of lattice sums, each corresponding to a facet of $\mathbf F \subset \mathbf P$, for which we then use Poisson summation in reverse, to recognize each lattice sum corresponding  to $\mathbf F$ as some familiar function back in the spatial domain. 

We will also deal with translations of functions, and for notational ease we define  $T_{x_0}(\xi) := \xi - x_0$ and recall the Translation Identity for Fourier transforms:
\[
\widehat{f \circ T_{x_0}}(\xi) = \hat{f}(\xi) e^{-2 \pi i \langle \xi, x_0 \rangle}.
\]

The following lemma enables us to evaluate our limiting infinite sums as finite sums, which will turn out to be either periodic Bernoulli polynomials or Dedekind-Rademacher sums.
	
	\begin{lem} \label{lem:Lem}
		If $f$ is a continuous function on the polytope $\mathcal{P}$ in $\mathbb{R}^d$ and is zero outside $\mathcal{P}$, then, for all $x \in \mathbb{R}^d$,
			\[ \lim_{\eps \to 0^+} (f \ast \Ge) (x) = f(x) \om_{\mathcal{P}}(x). \]
	\end{lem}

	\begin{proof}
		By direct computation, we have:
		\begin{align*}
			(f \ast \Ge) (x) &= \int_{\R^d} f(y) \Ge(x-y) dy \\
							&= \int_{\mathcal{P}} f(y) \Ge(y-x) dy \\
							&= \int_{T_{-x}(\mathcal{P})} f(u+x) \Ge(u) du \\
							&= \int_{\frac{1}{\sqrt{\eps}}T_{-x}(\mathcal{P})} f(x + v\sqrt{\eps}) \G_1(v) dv.
		\end{align*}
		where $T_{-x}(\mathcal{P})$ is the translation of $\mathcal{P}$ by the vector $-x$. Since the polytope $\mathcal{P}$ is closed and bounded, the function $f$ is uniformly continuous on $P$. Thus, when $\eps$ approaches $0$,  the convolution $(f \ast \Ge) (x)$ approaches the following limit:
\[ 
\lim_{\eps \to 0^+} (f \ast \Ge) (x) = f(x) \int_{K} \G_1(v)dv = f(x) \om_K(0) = f(x) \om_{\mathcal{P}}(x), 
\]
where $K$ is the tangent cone of $\mathcal{P}$ at the vertex $x$.
\end{proof}

We record here some easy facts concerning the Fourier transforms of some  Bernoulli polynomials that we will find useful. The Fourier transforms of the first and second Bernoulli polynomials, namely
$\hat{B}_1(\xi)$ and $\hat B_2(\xi)$, takes a particularly nice form when evaluated at integral frequencies:
\begin{equation*}
			\hat{B}_1(\xi) = \left\{
			\begin{array}{ll}
				\frac{1}{-2 \pi i \xi} & \textup{when } \xi \in \Z_{\neq 0}, \\
				0 & \textup{when } \xi = 0.
			\end{array}
			\right.
\end{equation*} 
\begin{equation*}
			\hat{B}_2(x) = \left\{
			\begin{array}{ll}
				\frac{1}{2 \pi^2 \xi^2} & \textup{when } \xi \in \Z_{\neq 0}, \\
				0 & \textup{when } \xi = 0.
			\end{array}
			\right.
\end{equation*}

\bigskip	
	\section{The quasi-coefficient $a_1(t)$}\label{Calc1}
	
We first recall the definition of $a_1(t)$:
	$$a_1(t) = t \sum_{i=1}^3 \lim_{\eps\to 0^+} \sum_{\xi\in (E_i^{\perp} \cap \Z^2) \setminus {(0,0)} } W_{(\Delta \to E_i)}(t\xi) \Geh(\xi).$$
	
	For each $i=1,2,3$, the limiting sum involving $E_i$ is the limit of a $1$-dimensional sum over $E_i^{\perp}$, excluding the origin. Notice that there are two primitive lattice vectors on the space $E_i^{\perp}$, and we let $v_i$ be the unique primitive lattice vector which is also an outward pointing normal vector to the face $E_i$.  Then, a lattice point $\xi$ on $E_i^{\perp}$ is an integer multiple of $v_i$, i.e. $\xi = \eta v_i$ for some integer $\eta$. We also need to pick an arbitrary point $\zeta_i$ on the face $E_i$. The rational and exponential weights of the chain $(\Delta \to E_i)$ read:
	
	\begin{align*}
		\mathcal{R}_{(\Delta \to E_i)}(t\xi) &= \frac{\vol(E_i)}{-2 \pi i}\frac{\langle t\xi,\frac{v_i}{\|v_i\|}\rangle}{\|t\xi\|^2} = t^{-1} \frac{\vol(E_i)}{-2 \pi i}\frac{\eta \|v_i\|}{\|\eta v_i\|^2} = t^{-1} \frac{\vol(E_i)}{\|v_i\|} \frac{1}{-2 \pi i \eta}, \\
		\mathcal{E}_{(\Delta \to E_i)}(t\xi) &= e^{-2 \pi i \eta (\langle v_i, \zeta_i \rangle t)}.\\
	\end{align*}
	
	For ease of notation, we put $p_i(t) = \langle v_i, \zeta_i \rangle t$ and obtain
		$$\mathcal{R}_{(\Delta \to E_i)}(t\xi) \mathcal{E}_{(\Delta \to E_i)}(t\xi) = t^{-1} \frac{\vol(E_i)}{\|v_i\|} \widehat{B_1 \circ T_{p_i(t)}}(\eta).$$
	
	Now we are in the position to compute the limiting sum involving $E_i$.
	
	\begin{align*}
	a_1(t)\quad &= \quad t \lim_{\eps\to 0^+} \sum_{\xi\in (E_i^{\perp} \cap \Z^2) \setminus {(0,0)} } W_{(\Delta \to E_i)}(t\xi) \Geh(\xi)\\
		&=\quad -\frac{\vol(E_i)}{\|v_i\|} \lim_{\eps\to 0^+} \sum_{\eta \in \Z_{\neq 0}} \widehat{(B_1 \circ T_{p_i(t)})}(\eta) \widehat{\G_{\eps\|v_i\|^2}}(\eta) \\
		&=\quad  -\frac{\vol(E_i)}{\|v_i\|} \lim_{\eps\to 0^+} \sum_{\eta \in \Z} \widehat{(B_1 \circ T_{p_i(t)})}(\eta) \widehat{\G_{\eps}}(\eta) \\
		&=\quad  -\frac{\vol(E_i)}{\|v_i\|} \lim_{\eps\to 0^+} \sum_{\eta \in \Z} \mathcal{F}\left\{(B_1 \circ T_{p_i(t)}) \ast \G_{\eps}\right\}(\eta) \\
		&=\quad  -\frac{\vol(E_i)}{\|v_i\|} \lim_{\eps\to 0^+} \sum_{n \in \Z} \big((B_1 \circ T_{p_i(t)}) \ast \G_{\eps}\big)(n) \\
		&=\quad  -\frac{\vol(E_i)}{\|v_i\|} \sum_{n \in \Z} (B_1 \circ T_{p_i(t)}) (n) \om_{[p_i(t),1+p_i(t)]}(n) \\
		&=\quad  -\frac{\vol(E_i)}{\|v_i\|} \bar{B}_1 (p_i(t)).
	\end{align*}
	
	The fourth equality follows from Poisson Summation Formula, while the fifth one is a result of Lemma \ref{lem:Lem}. The last equality can be easily derived by considering separately the cases when $p_i(t)$ is an integer or not. For the cases $i=2,3$, the sample points $\zeta_2,\zeta_3$ can be chosen to be the origin. Whence, $p_2(t),p_3(t)$ are identically zero and the limiting sums involving $E_2,E_3$ both vanish. When $i=1$, we have $v_1 = (k,h)$ and $\zeta_1 = (h,0)$, which implies $ p_i(t) = hkt $ and $\vol(E_1) = \|v_1\| = \sqrt{h^2 + k^2}$. Therefore, the quasi-coefficient $a_1(t)$ has the simple formula:
		$$a_1(t) = -\frac{\vol(E_1)}{\|v_1\|} \bar{B}_1(p_1(t)) = -\bar{B}_1(hkt).$$

\bigskip 
	
\section{The quasi-coefficient $a_0(t)$}\label{Calc2}
	With the appearance of complicated limits of  two-dimensional infinite lattice sums in the formula for $a_0(t)$, namely Equation (\ref{Eq:a_0}), it may be expected that the calculation of $a_0(t)$ is quite involved. Therefore, we split the computation into two parts. As we will see later, we can convert $a_0(t)$ into an aggregation of certain limiting sums. In this section, we only deal with the $1$-dimensional sums and other sums which vanish due to their intrinsic lattice symmetry. The next section will take care of the remaining unwieldy $2$-dimensional sums.
	
	To begin with, let us notice that the admissible set for each chain $(\Delta \to E_i \to V_j)$ only depends on $E_i$ and not on $V_j$. Moreover, the normal vectors $N_{E_i}(V_{j_1})$ and $N_{E_i}(V_{j_2})$, where $V_{j_1}$ and $V_{j_2}$ are two end-vertices of the edge $E_i$, are negatives of each other, which results in a nice relation between two rational weights 
	$$\mathcal{R}_{(\Delta \to E_i \to V_{j_1})}(\xi)=-\mathcal{R}_{(\Delta \to E_i \to V_{j_2})}(\xi).$$ 
	
	These observations suggest that we should combine the weights of the two chains $(\Delta \to E_i \to V_{j_1})$ and $(\Delta \to E_i \to V_{j_2})$. We make the following notations.
	
	\begin{align*}
		b_1(t) := t^2\lim_{\eps\to 0^+} \sum_{\xi\in\Z^2 \setminus E_1^{\perp}} (W_{(\Delta \to E_1 \to V_2)}(t\xi) + W_{(\Delta \to E_1 \to V_3)}(t\xi)) \Geh(\xi), \\
		b_2(t) := t^2\lim_{\eps\to 0^+} \sum_{\xi\in\Z^2 \setminus E_2^{\perp}} (W_{(\Delta \to E_2 \to V_1)}(t\xi) + W_{(\Delta \to E_2 \to V_3)}(t\xi)) \Geh(\xi), \\  
		b_3(t) := t^2\lim_{\eps\to 0^+} \sum_{\xi\in\Z^2 \setminus E_3^{\perp}} (W_{(\Delta \to E_3 \to V_1)}(t\xi) + W_{(\Delta \to E_3 \to V_2)}(t\xi)) \Geh(\xi).
	\end{align*}
	
	It turns out that $b_2(t)$ and $b_3(t)$ will vanish for all $t$ due to certain lattice symmetries in their summation domains. Let us consider $b_2(t)$ first and let $\xi = (\xi_1,\xi_2)$. Then,
	\begin{align*}
		\mathcal{R}_{(\Delta \to E_2 \to V_1)}(\xi) &= -\mathcal{R}_{(\Delta \to E_2 \to V_3)}(\xi) \\
			&= \frac{-1}{2 \pi i} \frac{\langle \proj_{\Delta}(\xi),N_{\Delta}(E_2) \rangle}{\|\proj_{\Delta}(\xi)\|^2} \frac{-1}{2 \pi i} \frac{\langle \proj_{E_2}(\xi),N_{E_2}(V_1) \rangle}{\|\proj_{E_2}(\xi)\|^2} \vol(V_1) \\	
			&= \frac{-1}{4\pi^2} \frac{\langle (\xi_1,\xi_2),(-1,0) \rangle}{\xi_1^2+\xi_2^2} \frac{\langle (0,\xi_2),(0,-1) \rangle}{\xi_2^2} = \frac{-1}{4\pi^2} \frac{\xi_1}{(\xi_1^2+\xi_2^2)\xi_2}, \\
		\mathcal{E}_{(\Delta \to E_2 \to V_1)}(\xi) &= e^{-2 \pi i \langle \xi, V_1 \rangle} = 1, \\
		\mathcal{E}_{(\Delta \to E_2 \to V_3)}(\xi) &= e^{-2 \pi i \langle \xi, V_3 \rangle} = e^{-2 \pi i k \xi_2},
	\end{align*}
	which implies
	\begin{align} \label{Eq:b_2}
		b_2(t)=\lim_{\eps\to 0^+} \sum_{\substack{\xi\in\Z^2 \\ \xi\notin E_2^{\perp}}} \mathcal{R}_{(\Delta \to E_2 \to V_1)}(\xi) \big( \mathcal{E}_{(\Delta \to E_2 \to V_1)}(t\xi) - \mathcal{E}_{(\Delta \to E_2 \to V_3)}(t\xi) \big) \Geh(\xi).
	\end{align}
	
	Suppose that $\xi_2$ is fixed. Then, as functions in $\xi_1$, $\mathcal{R}_{(\Delta \to E_2 \to V_1)}(\xi_1,\xi_2)$ is odd, while $\Geh(\xi)$ is even. Also, both $\mathcal{E}_{(\Delta \to E_2 \to V_1)}(t\xi)$ and $\mathcal{E}_{(\Delta \to E_2 \to V_3)}(t\xi)$ are constant in $\xi_1$. Finally, the summation domain $\Z^2 \setminus E_2^{\perp} = \{ (\xi_1,\xi_2) \in \Z^2 : \xi_2 \neq 0  \}$ is symmetric with respect to the line $\xi_1 = 0$ in the frequency plane. All these facts together imply that the $2$-dimensional sum in the above expression for $b_2(t)$ always vanishes for all nonzero $t$.
	
	Similarly, we now calculate the rational and exponential weights involved in the definition of $b_3(t)$.
	\begin{align*}
		\mathcal{R}_{(\Delta \to E_3 \to V_1)}(\xi) &= -\mathcal{R}_{(\Delta \to E_3 \to V_2)}(\xi) \\
			&= \frac{-1}{2 \pi i} \frac{\langle \proj_{\Delta}(\xi),N_{\Delta}(E_3) \rangle}{\|\proj_{\Delta}(\xi)\|^2} \frac{-1}{2 \pi i} \frac{\langle \proj_{E_3}(\xi),N_{E_3}(V_1) \rangle}{\|\proj_{E_3}(\xi)\|^2} \vol(V_2) \\	
			&= \frac{-1}{4\pi^2} \frac{\langle (\xi_1,\xi_2),(0,-1) \rangle}{\xi_1^2+\xi_2^2} \frac{\langle (\xi_1,0),(-1,0) \rangle}{\xi_1^2} = \frac{-1}{4\pi^2} \frac{\xi_2}{(\xi_1^2+\xi_2^2)\xi_1} \\
		\mathcal{E}_{(\Delta \to E_3 \to V_1)}(\xi) &= e^{-2 \pi i \langle \xi, V_1 \rangle} = 1, \\
		\mathcal{E}_{(\Delta \to E_3 \to V_2)}(\xi) &= e^{-2 \pi i \langle \xi, V_2 \rangle} = e^{-2 \pi i h \xi_1},
	\end{align*}
	which implies
	\begin{align} \label{Eq:b_3}
		b_3(t)=\lim_{\eps\to 0^+} \sum_{\substack{\xi\in\Z^2 \\ \xi\notin E_3^{\perp}}} \mathcal{R}_{(\Delta \to E_3 \to V_1)}(\xi) \big( \mathcal{E}_{(\Delta \to E_3 \to V_1)}(t\xi) - \mathcal{E}_{(\Delta \to E_3 \to V_2)}(t\xi) \big) \Geh(\xi). .
	\end{align}
	
	Again, notice that if we fix $\xi_1$ then $\mathcal{R}_{(\Delta \to E_3 \to V_1)}(\xi_1,\xi_2)$ is an odd function in $\xi_2$, while $\Geh(\xi)$ is an even one. Both $\mathcal{E}_{(\Delta \to E_3 \to V_1)}(t\xi)$ and $\mathcal{E}_{(\Delta \to E_3 \to V_2)}(t\xi)$ are independent of $\xi_2$. Finally, the summation domain $\Z^2 \setminus E_3^{\perp} = \{ (\xi_1,\xi_2) \in \Z^2 : \xi_1 \neq 0  \}$ is symmetric with respect to the line $\xi_2=0$ in the frequency plane. We therefore conclude that  $b_3(t)=0$ for all nonzero $t$ by using similar symmetry considerations as we did for $b_2(t)$ above. 
	
We now have $a_0(t) = b_1(t) + b_2(t) + b_3(t) = b_1(t)$.   We will further decompose $b_1(t)$ into 6 difficult limiting sums. But first, let us carry out the preliminary computation of $b_1(t)$.
	\begin{align*}
		\mathcal{R}_{(\Delta \to E_1 \to V_2)}(\xi) &= -\mathcal{R}_{(\Delta \to E_1 \to V_3)}(\xi) \\
			&= \frac{-1}{2 \pi i} \frac{\langle \proj_{\Delta}(\xi),N_{\Delta}(E_1) \rangle}{\|\proj_{\Delta}(\xi)\|^2} \frac{-1}{2 \pi i} \frac{\langle \proj_{E_1}(\xi),N_{E_1}(V_2) \rangle}{\|\proj_{E_1}(\xi)\|^2} \vol{V_2} \\
			&= \frac{-1}{4\pi^2} \frac{\langle \xi, N_{\Delta}(E_1) \rangle}{\|\xi\|^2} \frac{1}{\langle \xi,N_{E_1}(V_2) \rangle} \\
			&= \frac{-1}{4\pi^2} \frac{\langle (\xi_1,\xi_2),(\frac{h}{\sqrt{h^2+k^2}},\frac{k}{\sqrt{h^2+k^2}}) \rangle}{\xi_1^2+\xi_2^2} \frac{1}{\langle (\xi_1,\xi_2) , (\frac{h}{\sqrt{h^2+k^2}}, \frac{-k}{\sqrt{h^2+k^2}}) \rangle} \\
			&= \frac{-1}{4\pi^2} \frac{k\xi_1+h\xi_2}{(\xi_1^2+\xi_2^2)(h\xi_1-k\xi_2)}. \\
		\mathcal{E}_{(\Delta \to E_1 \to V_2)}(\xi) &= e^{-2 \pi i \langle \xi, V_2 \rangle} = e^{-2 \pi i h \xi_1}. \\
		\mathcal{E}_{(\Delta \to E_1 \to V_3)}(\xi) &= e^{-2 \pi i \langle \xi, V_3 \rangle} = e^{-2 \pi i k \xi_2}.
	\end{align*}
	Thus,
	\begin{align}
		a_0(t) &= b_1(t) = \lim_{\eps\to 0^+} \sum_{\substack{\xi\in\Z^2 \\ \xi\notin E_1^{\perp}}} \mathcal{R}_{(\Delta \to E_1 \to V_2)}(\xi) \big( \mathcal{E}_{(\Delta \to E_1 \to V_2)}(t\xi) - \mathcal{E}_{(\Delta \to E_1 \to V_3)}(t\xi) \big) \Geh(\xi) \nonumber\\
				&= \lim_{\eps\to 0^+} \sum_{\substack{\xi_1,\xi_2\in\Z \\ h\xi_1 \neq k\xi_2}} \frac{-1}{4\pi^2} \frac{k\xi_1+h\xi_2}{(\xi_1^2+\xi_2^2)(h\xi_1-k\xi_2)} (e^{-2 \pi i h \xi_1} - e^{-2 \pi i k \xi_2}) \Geh(\xi_1,\xi_2). \label{Eq:b_1}
	\end{align}
	Inspired by the computation of $b_2(t)$ and $b_3(t)$, we hope to decompose the lengthy formula of $b_1(t)$ into components with similar patterns. First, we break up the rational function in (\ref{Eq:b_1}) into partial fractions:
	\begin{align}
		\frac{(k\xi_1 + h\xi_2)}{(\xi_1^2+\xi_2^2)(h\xi_1 - k\xi_2)} &= \frac{k}{(h\xi_1 - k\xi_2)\xi_1} + \frac{\xi_2}{(\xi_1^2+\xi_2^2)\xi_1} \\
			&= \frac{h}{(h\xi_1 - k\xi_2)\xi_2} - \frac{\xi_1}{(\xi_1^2+\xi_2^2)\xi_2}.
	\end{align}
	Now the rational weights involved in $b_2(t)$ and $b_3(t)$ appear again. However, the summation domain now is not (but almost) symmetric to either the line $\xi_1=0$ or the line $\xi_2=0$. Therefore, we need to work a bit to `symmetrize' the summation domain. We also need to take care of the cases when $\xi_1$ or $\xi_2$ are zero, as separate sums. The limiting sum expression for $a_0(t)$ in Equation (\ref{Eq:b_1}) is therefore broken up into a sum of the following $6$ functions:
\begin{equation}
a_0(t) := c_1(t) + c_2(t) + c_3(t) + c_4(t) + c_5(t) + c_6(t),
\end{equation}
where
\begin{align}
		c_1(t) &:= \lim_{\eps\to 0^+} \sum_{\substack{\xi_1,\xi_2\in\Z \\ \xi_1 = 0, h\xi_1 \neq k\xi_2}} \frac{-1}{4\pi^2} \frac{(k\xi_1 + h\xi_2)}{(\xi_1^2+\xi_2^2)(h\xi_1 - k\xi_2)}  e^{-2\pi i h \xi_1 t} \Geh(\xi_1,\xi_2), \\
		c_2(t) &:= \lim_{\eps\to 0^+} \sum_{\substack{\xi_1,\xi_2\in\Z \\ \xi_1 \neq 0, h\xi_1 \neq k\xi_2}} \frac{-1}{4\pi^2} \frac{\xi_2}{(\xi_1^2+\xi_2^2)\xi_1} e^{-2\pi i h \xi_1 t} \Geh(\xi_1,\xi_2), \\
		c_3(t) &:= \lim_{\eps\to 0^+} \sum_{\substack{\xi_1,\xi_2\in\Z \\ \xi_1 \neq 0, h\xi_1 \neq k\xi_2}} \frac{-1}{4\pi^2} \frac{k}{(h\xi_1 - k\xi_2)\xi_1} e^{-2\pi i h \xi_1 t} \Geh(\xi_1,\xi_2), \\
		c_4(t) &:= -\lim_{\eps\to 0^+} \sum_{\substack{\xi_1,\xi_2\in\Z \\ \xi_2 = 0, h\xi_1 \neq k\xi_2}} \frac{-1}{4\pi^2} \frac{(k\xi_1 + h\xi_2)}{(\xi_1^2+\xi_2^2)(h\xi_1 - k\xi_2)}  e^{-2\pi i k \xi_2 t} \Geh(\xi_1,\xi_2), \\
		c_5(t) &:= -\lim_{\eps\to 0^+} \sum_{\substack{\xi_1,\xi_2\in\Z \\ \xi_2 \neq 0, h\xi_1 \neq k\xi_2}} \frac{-1}{4\pi^2} \frac{-\xi_1}{(\xi_1^2+\xi_2^2)\xi_2} e^{-2\pi i k \xi_2 t} \Geh(\xi_1,\xi_2), \\
		c_6(t) &:= -\lim_{\eps\to 0^+} \sum_{\substack{\xi_1,\xi_2\in\Z \\ \xi_2 \neq 0, h\xi_1 \neq k\xi_2}} \frac{-1}{4\pi^2} \frac{h}{(h\xi_1 - k\xi_2)\xi_2} e^{-2\pi i k \xi_2 t} \Geh(\xi_1,\xi_2). 
\end{align}
	
	The functions $c_1(t)$ and $c_4(t)$ take care of the cases when at least one of $\xi_1$ and $\xi_2$ is zero. Therefore, after substituting $\xi_1=0$ or $\xi_2=0$, these functions become $1$-dimensional sums, which can be calculated by the machinery introduced in Section \ref{Breakdown} and utilized in the previous section.
	\begin{align*}
		c_1(t) &= \lim_{\eps\to 0^+} \sum_{\xi_2 \in \Z_{\neq 0}} \frac{-1}{4\pi^2} \frac{h}{-k\xi_2^2}\Geh(0,\xi_2) \\
			&= \frac{h}{2k} \lim_{\eps\to 0^+} \sum_{\eta \in \Z} \hat{B_2}(\eta) \Geh(\eta) =  \frac{h}{2k} \lim_{\eps\to 0^+} \sum_{\eta \in \Z} \widehat{(B_2 \ast \Ge)}(\eta)\\
			&= \frac{h}{2k} \lim_{\eps\to 0^+} \sum_{n \in \Z} (B_2 \ast \Ge)(n) = \frac{h}{2k} \sum_{n \in \Z} B_2(n) \om_{[0,1]}(n) \\
			&= \frac{h}{2k} \big( B_2(0)\frac{1}{2} + B_2(1)\frac{1}{2} \big) = \frac{h}{12k}.
	\end{align*}
	
	By symmetry, the computation of $c_4(t)$ can be carried out in a completely similar manner, giving us: 
\begin{equation}
c_4(t) = \frac{k}{12h}.
\end{equation}
	
We move next to the calculation of $c_2(t)$ and $c_5(t)$. The key step is to `symmetrize' the summation domains.
	\begin{align*}
		c_2(t) &= \lim_{\eps\to 0^+} \Bigg( \sum_{\substack{\xi_1,\xi_2\in\Z \\ \xi_1 \neq 0}} - \sum_{\substack{\xi_1,\xi_2\in\Z \\ \xi_1 \neq 0, h\xi_1 = k\xi_2}} \Bigg) \frac{-1}{4\pi^2} \frac{\xi_2}{(\xi_1^2+\xi_2^2)\xi_1} e^{-2\pi i h \xi_1 t} \Geh(\xi_1,\xi_2).
	\end{align*}
	Notice that summation domain $S=\{\xi_1,\xi_2\in\Z : \xi_1 \neq 0\}$ is symmetric about the line $\xi_2=0$. Hence, using the same argument as in the computation of the functions $b_2(t)$ and $b_3(t)$ in Section \ref{Calc1}, the sum over $S$ vanishes for all nonzero $t$ and all positive $\eps$. 
	
	The remaining sum now becomes a $1$-dimensional sum (lying in the $2$-dim'l plane) whose summation domain can be parametrized as 
		$$\{(\xi_1,\xi_2) \in\Z^2 : \xi_1 \neq 0 \textup{ and } h\xi_1 = k\xi_2\} 
		= \{  (\eta k, \eta h):  \ \eta\in\Z_{\neq 0}\}.$$
	Under this parametrization, the formula for $c_2(t)$ becomes:
	\begin{align*}
		c_2(t) &= -\lim_{\eps\to 0^+} \sum_{\eta\in\Z_{\neq 0}} \frac{-1}{4\pi^2}\frac{\eta h}{(\eta^2 h^2+\eta^2 k^2) \eta k} e^{-2 \pi i hkt \eta} \hat{\G}_{\eps(h^2+k^2)}(\eta) \\
			&= \frac{h}{2k(h^2+k^2)} \lim_{\eps\to 0^+}\sum_{\eta\in\Z} \widehat{(B_2 \circ T_{hkt})} \Geh(\eta) \\
			&= \frac{h}{2k(h^2+k^2)} \lim_{\eps\to 0^+}\sum_{\eta\in\Z} \mathcal{F}\left\{(B_2 \circ T_{hkt}) \ast \Ge\right\}(\eta)\\
			&= \frac{h}{2k(h^2+k^2)} \lim_{\eps\to 0^+}\sum_{n\in\Z} ((B_2 \circ T_{hkt}) \ast \Ge)(n) \\
			&= \frac{h}{2k(h^2+k^2)} \sum_{n\in\Z} (B_2 \circ T_{hkt})(n) \om_{[hkt,1+hkt]}(n)
	\end{align*}
	By a little consideration of the two separate cases when $hkt$ is an integer or not, it is easy to show that the last sum is equal to $\bar{B}_2(hkt)$. Recall that $\bar{B}_2(x)$ is the second periodic Bernoulli polynomial. Therefore, the explicit formula for $c_2(t)$ is:
		$$c_2(t) = \frac{h}{2k(h^2+k^2)} \bar{B}_2(hkt).$$
	Due to the symmetry in the parameters and indices, we can compute $c_5(t)$ using the same method that was used for $c_2(t)$ above, and we obtain the following formula:
		$$c_5(t) = \frac{k}{2h(h^2+k^2)} \bar{B}_2(hkt).$$
		
We remark that thus far all of our computations were made for the general case of $h,k \in \mathbb R$, and thus any real unimodular triangle, and we did not require the assumption of their integrality or even rationality.   In the next section, however, in order to simplify the computation of $c_3(t)$ and $c_6(t)$, we will restrict attention to the case $h,k \in \mathbb Z$.

\bigskip 

\section{Proof of the Main Theorem}\label{Calc3}

In order to complete the proof of the Main Theorem, we must evaluate the more complicated 
$2$-dimensional lattice sums $c_3(t)$ and $c_6(t)$.  Again by symmetry of the indices, we only need to compute
$c_3(t)$, and the formula for $c_6(t)$ will follow easily.  For the convenience of the reader we restate the definition of $c_3(t)$ here:
	
\begin{equation}
c_3(t) := \lim_{\eps\to 0^+} \sum_{\substack{\xi_1,\xi_2\in\Z \\ \xi_1 \neq 0, h\xi_1 \neq k\xi_2}} \frac{-1}{4\pi^2} \frac{k}{(h\xi_1 - k\xi_2)\xi_1} e^{-2\pi i h \xi_1 t} \Geh(\xi_1,\xi_2).
\end{equation}
	
	In order to tackle this rather delicate limit, we will show that the rational function in the summand is in fact the Fourier transform of a compactly supported function when the frequency $\xi=(\xi_1,\xi_2)$ is an integer point.  We remark that it is not possible to take the limit inside the lattice sum, because the ensuing sum will be formally divergent.

	Therefore, the plan is to follow the same method as in the previous sections, so that we may reduce the $2$-dimensional infinite sum to a finite sum over a parallelogram. When $ht$ is an integer, the vertices of the parallelogram are integer points and hence we need to take into account the solid angles at those points, which are closely related to the solid angles at the vertices of the triangle $\Delta$. In this case, the function $c_3(t)$ is the negative of the sum of the solid angle at $V_3$ of $\Delta$ and the Dedekind sums $s(h,k)$. In the other case when $ht$ is not in $\Z$, the vertices of the parallelogram will not be lattice points and the formula of $c_3(t)$ is just the nagative of the Dedekind-Rademacher sum $s(h,k;ht,0)$.

Here we define a two-dimensional analogue of the Bernoulli polynomial, which is compactly supported on the unit square $[0,1]^2$, and which we will need in order to analyze the quasi-coefficient $a_0(t)$.   First, we recall our definition of the real-valued one-dimensional Bernoulli polynomial, restricted to be compactly supported on the closed unit interval:
	 
	 \begin{equation*}
			B_1(x) = \left\{
			\begin{array}{ll}
				x - \frac{1}{2}       & \textup{when } x \in [0,1], \\
				0         & \textup{otherwise.}
			\end{array}
			\right.
		\end{equation*}

We define the following product of two such Bernoulli polynomials, which is therefore compactly supported on $[0,1]^2$:
	 \begin{equation*}
			\mathcal B(x,y) = \left\{
			\begin{array}{ll}
				B_1(x) B_1(y)      & \textup{when } (x,y)  \in [0,1]^2, \\
				0         & \textup{otherwise.}
			\end{array}
			\right.
		\end{equation*} 

First, the $1$-dimesional Fourier transform of $B_1(x)$ is retrieved easily by one application of integration by parts, and gives us:	

		\begin{align*}
			\hat{B}_1(n) &= \int_0^1 \left(x - \frac{1}{2}    \right) e^{2\pi i xn} dx \\
			&=\frac{1}{2} \frac{e^{2\pi i n} + 1}{2\pi i n} -  \frac{ e^{2\pi i n} - 1}{(2\pi i n)^2},
		\end{align*} 
	which is valid for all $n \in \mathbb R_{\not= 0}$.   When $n=0$, we have $\hat{B}_1(0)=0$.  Therefore, for integral frequencies $n \in \mathbb Z_{\neq 0}$, we get the particularly pleasing 
form $\hat{B}_1(n) = \frac{1}{2\pi i }  \frac{1}{n}$.   Similarly, we have:

		 \begin{align*}
		\hat{\mathcal B}(m,n) &=  \hat{B}_1(m) \hat{B}_2(n)\\
			&= \left(
			\frac{1}{2} \frac{e^{2\pi i m} + 1}{2\pi i m} -  \frac{ e^{2\pi i m} - 1}{(2\pi i m)^2}
			\right)
			\left(
			\frac{1}{2} \frac{e^{2\pi i n} + 1}{2\pi i n} -  \frac{ e^{2\pi i n} - 1}{(2\pi i n)^2}
			\right),
		\end{align*} 		
which again has a pleasing form when it is evaluated at integer vectors $(m,n)$.  Namely, we get:
\begin{equation}
\hat{\mathcal B}(m,n) =  \frac{1}{(2\pi i)^2}  \frac{1}{mn},
\end{equation}		
for all integer vectors $(m,n) \in \mathbb Z^2$.   
Next, we twist the Fourier transform by a linear transformation.  In other words, let  $M=\left(\begin{array}{ll} m & n \\ p & q \end{array}\right)$ be any matrix in $\textup{GL}(2,\Z)$, and let $M^{-T}$ be its inverse transpose.  We recall the standard
identity 
\[
\widehat{ f \circ M^{-T}}(\xi_1, \xi_2) =  |\det(M)| (\hat{f} \circ M) (\xi_1,\xi_2), 
\]
valid for all real vectors $(\xi_1, \xi_2) \in \mathbb R^2$.   Applying this identity to the function $f := \mathcal B$, we have arrived at the following result for integer vectors.

\begin{lem} \label{lem:Lem2}
For all integer vectors $(\xi_1, \xi_2) \in \mathbb Z^2$, we have
\begin{equation}		
\widehat{     \mathcal{B} \circ M^{-T}    }(\xi_1,\xi_2) =  \frac{|\det(M)|}{(2\pi i)^2} 
			\frac{1}{  (m\xi_1 + n\xi_2)(p\xi_1+q\xi_2)   }.
\end{equation}
\end{lem}

\noindent
We now use Lemma \ref{lem:Lem2} with the particular integral matrix $M = \left(\begin{array}{ll} h & -k \\ 1 & 0 \end{array}\right)$.   We further define
 $\mathcal{B}^*:=(\mathcal{B} \circ M^{-T} \circ T_{(ht,0)})$, a compactly supported function on the closed parallelogram 
 $H=(T_{(ht,0)}^{-1} \circ M^T)([0,1]^2)$, which has four vertices at $(ht,0)$, $(ht+1,0)$, $(ht+h,-k)$ and $(ht+h+1,-k)$. We have
	\begin{align*}
		c_3(t) &= \lim_{\eps\to 0^+} \sum_{(\xi_1,\xi_2)\in\Z^2} \mathcal{F}\left\{\mathcal{B} \circ M^{-T} \circ T_{(ht,0)} \right\} (\xi_1,\xi_2) \Geh(\xi_1,\xi_2) \\
			&= \lim_{\eps\to 0^+} \sum_{(\xi_1,\xi_2)\in\Z^2} \mathcal{F}\left\{(\mathcal{B} \circ M^{-T} \circ T_{(ht,0)}) \ast \Ge\right\}(\xi_1,\xi_2) \\
			&= \lim_{\eps\to 0^+} \sum_{(x_1,x_2)\in\Z^2} \left((\mathcal{B} \circ M^{-T} \circ T_{(ht,0)}) \ast \Ge\right)(x_1,x_2).
	\end{align*}

		Now, by Lemma \ref{lem:Lem}, the limiting sum in the latter formula of $c_3(t)$ is the finite sum over lattice points of the parallelogram $H$, each of whose summands is the product of the function $\mathcal{B}^*$ and the solid angle subtended by $H$ at each lattice point.  We rewrite the expression of $c_3(t)$ as follows:
	\begin{align}
		c_3(t) = \sum_{-k \leq x_2 \leq 0} \quad \sum_{ht-\frac{h}{k}x_2 \leq x_1 \leq ht-\frac{h}{k}x_2+1} \mathcal{B}^*(x_1,x_2) \om_H(x_1,x_2). \label{Eq:c_3}
	\end{align}
	
	When $-k < x_2 < 0$, there is either one or two lattice points in $H$ whose ordinates are $x_2$. If there are two lattice points, the solid angles subtended by $H$ at those points are both equal to $1/2$, but the valuations of $\mathcal{B}^*$ at those points are negatives of each other. Thus, the inner sum in (\ref{Eq:c_3}) vanishes in that case. In the other case that there is exactly one lattice point, that unique point has ordinate $\lfloor ht-\frac{h}{k}x_2 \rfloor$ and is also an interior point of $H$, which implies that the solid angle there is simply $1$. Therefore, after direct computation, the inner sum in (\ref{Eq:c_3}) becomes
		$$\bar{B}_1 \left(\frac{1}{k}x_2\right) \bar{B}_1 \left(ht-\frac{h}{k}x_2\right).$$
	Notice that this formula also agrees with the result in the previous case when there are two lattice points inside $H$ that have ordinate $x_2$. Hence, the sum of the former expression over all $-k < x_2 < 0$ is precisely the negative of the Dedekind-Rademacher sum $s(h,k;ht,0)$.
	
	The last two cases, namely $x_2=-k$ and $x_2=0$, are treated similarly by separating the cases when $ht$ is an integer or not. When $ht$ is not an integer, the evaluations of the inner sum of (\ref{Eq:c_3}) in these two cases are negatives of each other and thus cancel each other (one is $\bar{B}_1(ht)/4$, the other is $-\bar{B}_1(ht)/4$). When $ht$ is an integer, we need to count the solid angles at the vertices of the parallelogram $H$. It turns out that, in both of these cases, when $x_2=-k$ and when $x_2=0$, the inner sum of (\ref{Eq:c_3}) is equal to  $-\arctan(h/k)/(4\pi) \mathbf{1}_{\Z}(ht)$. Recall that the indicator function $\mathbf{1}_{\Z}(x)$ is equal to $1$ when $x$ is an integer and $0$ when $x$ is not an integer. Therefore, the use of the former indicator function helps unify the two sub-cases, as follows. We obtain the following concise formula for $c_3(t)$:
		$$c_3(t) = -s(h,k;ht,0)-2\frac{\arctan(h/k)}{4\pi} \mathbf{1}_{\Z}(ht) = -s(h,k;ht,0)-\frac{\arctan(h/k)}{2\pi} \mathbf{1}_{\Z}(ht).$$
		
	By symmetry, we are able to compute a similar formula for $c_6(t)$:
		$$c_6(t) = -s(k,h;kt,0)-\frac{\arctan(k/h)}{2\pi} \mathbf{1}_{\Z}(kt).$$
	
	Finally, we obtain the desired explicit formula for $a_0(t)$ which is the sum of all $c_i(t)$ for $i=1,...,6$:
	\begin{align*}
		a_0(t) &= \frac{h}{12k} + \frac{h}{2k(h^2+k^2)} \bar{B}_2(hkt) - s(h,k;ht,0)-\frac{\arctan(h/k)}{2\pi} \mathbf{1}_{\Z}(ht) \\
			&\quad + \frac{k}{12h} + \frac{k}{2h(h^2+k^2)} \bar{B}_2(hkt) - s(k,h;kt,0)-\frac{\arctan(k/h)}{2\pi} \mathbf{1}_{\Z}(kt) \\
			&= \frac{1}{2hk}\left(\bar{B}_2(hkt)+\frac{h^2+k^2}{6}\right) - s(h,k;ht,0) - s(k,h;kt,0) \\
			&\quad - \frac{\arctan(h/k)}{2\pi} \mathbf{1}_{\Z}(ht) - \frac{\arctan(k/h)}{2\pi} \mathbf{1}_{\Z}(kt).
	\end{align*}	
\noindent 
The proof of the Main Theorem is now complete.

%Section 6:  

\bigskip 
	
	\section{Retrieving some classical results from the main theorem}\label{Apps}
	From the previous sections, it is easy to see that the quasi-coefficients of $A_{\Delta}(t)$ are periodic in $t$ with period $1$, as proved in \cite{DLR}.  In the special case that 
	$t$ is an integer, our formula for the solid-angle sum $A_{\Delta}(t)$ simplifies:
	\begin{align*}
		A_{\Delta}(t) &= \frac{hk}{2}t^2 - 0 \cdot t + \frac{1}{2hk} \left(\frac{1}{6} + \frac{h^2+k^2}{6}\right) - s(h,k) - s(k,h) \\
					&\quad - \frac{\arctan(h/k)}{2\pi} - \frac{\arctan(k/h)}{2\pi} \\
					&= \frac{hk}{2}t^2 + \frac{1}{12} \left( \frac{h}{k} + \frac{1}{hk} + \frac{k}{h} \right) - \frac{1}{4} - s(h,k) - s(k,h),
	\end{align*}
for $t \in \Z_{>0}$.	Macdonald has shown that  for integer values of $t$, $A_{\Delta}(t)$ is exactly the volume of $t\Delta$, both $a_1(t)$ and $a_0(t)$ vanish, and that $A_{\Delta}(0) = 0$.   The statement that the last quasi-coefficient term vanishes, namely that $a_0(t) = 0$,  is therefore 
equivalent to the classical reciprocity law for the Dedekind sums.  Thus, we have just found another proof of the famous reciprocity law for the Dedekind sums:
\[
s(h,k) + s(k,h) =  \frac{1}{12} \left( \frac{h}{k} + \frac{1}{hk} + \frac{k}{h} \right) - \frac{1}{4}.
\]

This proof is similar to the approach taken in \cite{Matt}.  It is also obvious that the above formula satisfies the Generalized Macdonald's Reciprocity, introduced in \cite{DeR} and  \cite{DLR}.
	
	There is an intrinsic connection between the theory of Ehrhart sums and that of Macdonald's solid-angle sums. Recall that the Ehrhart sum $L_S(t)$ of a subset $S$ of $\R^2$ is the number of lattice points in the dilate $tS$.
		$$L_S(t):=\#\{\Z^2 \cap tS\}.$$
	Suppose $\mathcal{P}$ is a closed lattice polygon. The difference $L_P(t) - A_P(t)$ in the $2$-dimensional case is rather simple. It is just half the number of lattice points on the edges of $tP$ (every vertex is counted twice) minus the sum of the solid angles at each vertex of $tP$, whenever that vertex is a lattice point. In other words, if $\mathcal{P}$ is a polygon with $n$ vertices $V_1,...,V_n$ and $n$ closed edges $E_1,...,E_n$, then 
		$$L_{\mathcal{P}}(t) = A_{\mathcal{P}}(t) + \frac{1}{2} \sum_{i=1}^n L_{E_i}(t) - \sum_{i=1}^n \mathbf{1}_{\Z^2}(tV_i) \om_P(V_i).$$
	We will have no problem handling the last sum in the above expression. For the other sum, let us consider each closed edge separately. If the line $E_i$ contains the origin, all dilates of $E_i$ will lie on the same $1$-dimensional subspace. That subspace has a simple lattice structure generated by a primitive lattice vector. Therefore, $L_{E_i}(t)$  is easy to compute.     If the line $E_i$ does not contain the origin, we can employ the machinery introduced in Section $\ref{Breakdown}$ to break down the general case to the case when $E_i$ is the hypotenuse $E_1$ of the triangle $\Delta$. However, because the additive property of Ehrhart sums is  not as simple as that of the solid-angle sum, we must take care of the one-point intersections between the segments in the unimodular decomposition of $E_i$. This obstacle is also not difficult to overcome, and we now derive an exact formula for $L_{E_1}(t)$ of the hypotenuse $E_1$ of $\Delta$ based on a theorem of T. Popoviciu. 
	
	\begin{popo} If $a$ and $b$ are coprime positive integers, then, for every natural number $n$, the number of decomposition of $n$ into a linear combination of $a$ and $b$ whose coefficients are two positive integers is
	\begin{align*}
		p_{\{a,b\}}(n) &:= \# \left\{(x,y)\in\Z^2:x,y\geq 0,ax+by=n\right\}\\
						&=\frac{n}{ab} - \left\{\frac{b^{-1}n}{a}\right\} - \left\{\frac{a^{-1}n}{b}\right\} + 1,
	\end{align*}
	where $a^{-1}$ and $b^{-1}$ are two integers satisfying $a^{-1}a\equiv 1 \textup{ mod } b$ and $b^{-1}b\equiv 1 \textup{ mod } a$.
	\end{popo}
	
	Note that $(x,y)$ is a lattice point of the segment $tE_1$ if and only if $x$ and $y$ are two positive integers satisfying $\frac{x}{h} + \frac{y}{k} = t$, or equivalently 
	$kx + hy = hkt$.       Therefore, for any nonzero real number $t$, we have
		$$L_{E_1}(t) = \mathbf{1}_{\Z}(hkt) p_{\{h,k\}}(|hkt|).$$
	
	This result therefore implies that $L_{\mathcal{P}}(t)$ can be computed in polynomial time for any closed lattice polygon $\mathcal{P}$ and any nonzero real number $t$.
	
	In the specific case of the right-angled triangle $\Delta$, we have $L_{E_3}(t) = \lfloor ht \rfloor + 1$ and $L_{E_2}(t) = \lfloor kt \rfloor + 1$. By handling carefully the cases when $ht$ or $kt$ is an integer, and putting together all of the above considerations in this section, we obtain the following result, which is an extension of the Ehrhart quasi-polynomial of $\Delta$ to all real dilation parameters.  This is the main result of this section.
	
\begin{cor}\label{Cor}  For all real positive values of $t$, we have
\begin{align}
		L_{\Delta}(t) &= A_{\Delta}(t) + \frac{1}{2} (L_{E_1}(t) + L_{E_2}(t) + L_{E_3}(t)) \\    
					&\qquad - \frac{1}{4} - \mathbf{1}_{\Z}(ht) \frac{\arctan(k/h)}{2\pi} - \mathbf{1}_{\Z}(kt) \frac{\arctan(h/k)}{2\pi} \\    \label{lastformula}
					&= \frac{1}{2hk} \lfloor hkt \rfloor (\lfloor hkt \rfloor + 1) + \frac{1}{2} (\lfloor ht \rfloor + \lfloor kt \rfloor) + \frac{3}{4} + \frac{1}{12} \left( \frac{h}{k} + \frac{1}{hk} + \frac{k}{h} \right)\\
					&\qquad - s^*(h,k;ht,0) - s^*(k,h;kt,0) - \frac{1}{2}(\{ht\}+\{kt\}),
\end{align}
where $s^*(h,k;y,x)$ is the modified Dedekind-Rademacher sum	
$$
s^*(h,k;y,x) := \sum_{r \textup{ mod } k} \bar{B}_1^*\left( h\frac{r+x}{k} + y \right) \bar{B}_1^*\left( \frac{r+x}{k} \right).
$$
Here $\bar{B}_1^*(x) = x - \lfloor x \rfloor - \frac{1}{2}$ is defined for all real $x$. Note that $\bar{B}_1^*(x)$ only differs from $\bar{B}_1(x)$ at the integer points. 
\end{cor}

\bigskip
	
\section{ Additional remarks}  

As easily seen from the above formula \eqref{lastformula}, $L_{\Delta}(t)$ is a right-continuous function. Also, we can show that the difference between the right-limit and the left-limit at points $t=\frac{n}{hk}$, where $n$ is a positive integer, is exactly $p_{\{h,k\}}(n)$, which is the number of lattice points on the edge $tE_1$. 
	
We remark that a rather surprising fact about Corollary \ref{Cor} is that this formula is piecewise constant. This fact can be deduced from a result in \cite{Knu} by D. Knuth. In that article, Knuth uses the notation
		$$\sigma(h,k,c):=12\sum_{r \textup{ mod } k}\bar{B}_1\left(h\frac{r+c}{k} \right) \bar{B}_1 \left(\frac{r}{k} \right) = 12 s(h,k;\frac{c}{k},0),$$
	for any two relatively prime integers $h,k$ and any real number $c$. We may modify the above formula a little bit and define another function,
		$$\sigma^*(h,k,c):=12\sum_{r \textup{ mod } k} \bar{B}_1^*\left( h\frac{r+c}{k} \right) \bar{B}_1^*\left( \frac{r}{k} \right) = 12 s^*(h,k;\frac{c}{k},0).$$
	The relation between two functions $\sigma$ and $\sigma^*$ depends on whether or not $c$ is an integer. If $c$ is not an integer, then
		$$\sigma^*(h,k,c) = \sigma(h,k,c) - 6 \bar{B}_1^*\left(\frac{c}{k}\right),$$
	otherwise
		$$\sigma^*(h,k,c) = \sigma(h,k,c) - 6 \bar{B}_1^*\left(\frac{c}{k}\right) + 6 \bar{B}_1\left( \frac{h^{-1}c}{k}\right), $$
	where $h^{-1}$ denotes an integer satisfying $h^{-1}h \equiv 1 \textup{ mod } k$. For any integer $n$ and any real number $0<\nu<1$, Lemma 1 in \cite{Knu} asserts that
		$$\sigma(h,k,n+\nu) = \sigma(h,k,n) + 6 \bar{B}_1\left( \frac{h^{-1}n}{k}\right), $$
	Using the above relations between two functions $\sigma$ and $\sigma^*$, this result can be restated that
		$$\sigma^*(h,k,n+\nu) + 6 \bar{B}_1^*\left(\frac{n+\nu}{k}\right) = \sigma^*(h,k,n) + 6 \bar{B}_1^*\left(\frac{n}{k}\right),$$
	or equivalently,
		$$s^*(h,k;\frac{n+\nu}{k},0) + \frac{1}{2}\left\{ \frac{n+\nu}{k} \right\} = s^*(h,k;\frac{n}{k},0) + \frac{1}{2}\left\{ \frac{n}{k} \right\}.$$
	This relation asserts that the function $L_{\Delta}(t)$ is always constant on the interval $\left[\frac{n}{hk},\frac{n+1}{hk}\right)$. Although the fact $L_{\Delta}(t)$ is piecewise constant, for postive real numbers $t$, is rather obvious from its geometric definition, the above argument shows the formula in Corollary \ref{Cor} is piecewise constant on the whole range of real numbers.
	
	The concise formula of $L_{\Delta}(t)$ given in Corollary \ref{Cor}  can be verified in many specific cases. First, we consider $t$ to be an integer. Then, the value of $L_{\Delta}(t)$ is $\frac{1}{hk}t^2 + \frac{1}{2}(h+k+1)t + 1$, satisfying Pick's theorem. Secondly, we may consider the more general case where $t = \frac{n}{hk}$ for some positive integer $n$. In this case, Corollary \ref{Cor} reduces to
	\begin{align*}
		L_{\Delta}(t) &= \frac{n^2}{2hk} + \frac{n}{2}\left(\frac{1}{h} + \frac{1}{k} + \frac{1}{hk}\right) + \frac{1}{4} + \frac{1}{12} \left(\frac{h}{k} + \frac{k}{h} + \frac{1}{hk}\right)\\
			& - s^*(h,k;\frac{n}{k},0) - s^*(k,h;\frac{n}{h},0) - \bar{B}_1^*\left(\frac{n}{h}\right) - \bar{B}_1^*\left(\frac{n}{k}\right),
	\end{align*}
	which matches perfectly with Proposition 3.5 in Beck and Robins  \cite{BeR2}.
	
Although Corollary \ref{Cor} may also be proved by combining the work of Knuth \cite{Knu} with the work of Beck and Robins \cite{BeR2}, here we stress a more unified approach to deriving it and other more general results for real dilations.  
	
	Some of our methods, in particular the Poisson summation approach, arose in the paper \cite{DiR}.  A fascinating and completely different method for studying Ehrhart theory was initiated in \cite{Pom}, using Toric varieties and Todd classes.  Earlier, McMullen \cite{Mcm1, Mcm2} studied very interesting extensions of Ehrhart theory, using valuation theory, and although valuations do not yet seem to give formulas for the coefficients of the solid angle polynomial, they do provide very beautiful structural information about these polynomials.  Our solid angle polynomial $A_P(t)$ is an example of a {\it simple valuation} of polytopes in that context, which means that the valuation vanishes on lower-dimensional polytopes, making computations easier because we do not have to worry about the lower-dimensional intersections that do arise in Ehrhart polynomials of closed polytopes.  
	
We note that Hardy and Littlewood studied in \cite{HL1, HL2}  the number of integer points in integer dilations of the right triangle $\Delta$ which we call a simple pointed triangle. Their methods include the study of the Barnes zeta function, which is an interesting but different route.

\bigskip 

\section{Future directions}
The approach taken here can indeed be extended to higher dimensions, although it may be much more difficult to transform the infinite lattice sums (arising from Poisson summation and our discrete Stokes' formula) into  closed forms in the higher dimensional case.   We have used some partial fraction identities in section \ref{Calc2}  to transform the more challenging infinite lattice sums arising from Poisson summation into closed-forms, in terms of Bernoulli polynomials.   Does this method extend to higher dimensions?  In other words, is it possible to find a systematic partial fraction approach, which combines with our output from Poisson summation, to always give some closed-form expressions in terms of higher-dimensional analogues of Dedekind sums?

On the one hand, higher dimensional polytopes pose the additional difficulty that the number of combinatorial chains that come from the face poset of $\mathcal P$ increases exponentially with the dimension.  Indeed, even for a $d$-dimensional simplex, there are $(d+1)!$ chains in the face poset that we considered.  On the other hand, the combinatorial flavor of the face poset which enters the whole picture may play an interesting and non-trivial role in the future development of these methods.

Recently, the very interesting work of Cristofaro-Gardiner, Li, and Stanley \cite{CLS}, analyzes the Ehrhart sum for the simple pointed triangle  $\Delta$ whose hypotenuse is allowed to have irrational slope.  Their combinatorial methods  handle integer dilation parameters, but there is also strong hope that our methods here, together with further work, may extend their results to all real dilation parameters and to other triangles.  Indeed, most of our computations will remain unchanged when $h,k$ are allowed to be irrational real numbers,  except for the intricate computation
 of $c_3(t)$ and $c_6(t)$.   It will be interesting to see how much further one can go, with either the combinatorial techniques, or the analytic techniques, or both.

\end{document}